\documentclass[11pt,reqno]{amsart}

\usepackage{amsmath,amsthm,a4}
\usepackage[boxed]{algorithm}
\usepackage[noend]{algorithmic}
\usepackage[latin1]{inputenc}
\usepackage[T1]{fontenc}
\usepackage{graphicx}

\theoremstyle{plain}
\newtheorem{thm}{Theorem}

\newtheorem{lem}{Lemma}[section]
\newtheorem{prop}[lem]{Proposition}

\theoremstyle{definition}
\newtheorem{definition}[lem]{Definition}

\newcommand\N{\mathbf N}
\newcommand\Z{\mathbf Z}
\newcommand\pt[2]{(#1,#2)}
\newcommand\LGV{Lindstr\"om-Gessel-Viennot}
\newcommand\Sym{\mathcal S}
\newcommand\sg{\mathop\mathrm{sg}}
\newcommand\supp{\mathop\mathrm{supp}}
\newcommand\Pathfam{\mathop\mathrm{Pathfam}}
\let\thru=\cap

\newcommand\transp[1]{#1^\top}
\newcommand\ip{i+1\binoppenalty10000 }
\newcommand\imn{i-1\binoppenalty10000 }
\newcommand\kp{k+1\binoppenalty10000 }
\newcommand\disj[1]{\textit{untangle}(#1,k)}
\newcommand\clify[1]{\textit{cliffify}(#1,k)}

\newcommand\setof[2]{\{\,#1\mid#2\,\}}
\newcommand\union{\cup}
\newcommand\smallmat[4]{\bigl(\genfrac{}{}{0pt }1{#1~#2}{#3~#4}\bigr)}

\newcommand\FROM{\textbf{ from }}
\newcommand\UPTO{\textbf{ to }} % variable status of \TO makes it unusable
\newcommand\DOWNTO{\textbf{ down to }}
\algsetup{indent=16pt}
\floatname{algorithm}{Procedure}

\begin{document}
\title[An Aztec diamond bijection by combing lattice paths]
  {A bijection proving the Aztec diamond theorem by combing lattice paths}
\author{Frédéric Bosio \and Marc van Leeuwen}
\address{Université~de~Poitiers,~Mathématiques,
	11~Boulevard~Marie~et~Pierre~Curie
	BP~30179
	86962~Futuroscope~Chasseneuil~Cedex
	France}

\email{Frederic.Bosio@math.univ-poitiers.fr}
\email{Marc.van-Leeuwen@math.univ-poitiers.fr}

\begin{abstract}
  We give a bijective proof of the Aztec diamond theorem, stating that there
  are $2^{n(n+1)/2}$ domino tilings of the Aztec diamond of order~$n$. The
  proof in fact establishes a similar result for non-intersecting families of
  $n+1$ Schr\"oder paths, with horizontal, diagonal or vertical steps, linking
  the grid points of two adjacent sides of an $n\times n$ square grid; these
  families are well known to be in bijection with tilings of the Aztec
  diamond. Our bijection is produced by an invertible ``combing'' algorithm,
  operating on families of paths without non-intersection condition, but
  instead with the requirement that any vertical steps come at the end of a
  path, and which are clearly $2^{n(n+1)/2}$ in number; it transforms them into
  non-intersecting families.

\end{abstract}
\maketitle

\section{Introduction}
The term ``Aztec diamond'', introduced by Elkies, Kuperberg, Larsen and Propp
\cite{EKLP}, refers to a diamond-shaped set of squares in the plane, obtained
by taking a triangular array of squares aligned against two perpendicular
axes, and completing it with its mirror images in those two axes; the
\emph{order} of the diamond is the number of squares along each of the sides
of the triangular array. Their main result concerns counting the number of
domino tilings (i.e., partitions into subsets of two adjacent squares) of the
Aztec diamond.

\begin{thm}[Aztec diamond theorem]
  There are exactly $2^{\binom{n+1}2}$ domino tilings of the Aztec diamond of
  order $n$.
\end{thm}

This result has been proved in various manners; the original article alone
gives four different proofs, all closely related to a correspondence that it
establishes between the domino tilings and certain pairs of alternating sign
matrices. Domino tilings of an order $n$ Aztec diamond can be brought into a
straightforward bijection with non-intersecting families of $n+1$ lattice
paths between two adjacent sides of an $n\times n$ square grid, using
horizontal, diagonal or vertical steps, as is illustrated in
figure~\ref{path-tiling}.
\begin{figure}
  \includegraphics[width=0.9\textwidth]{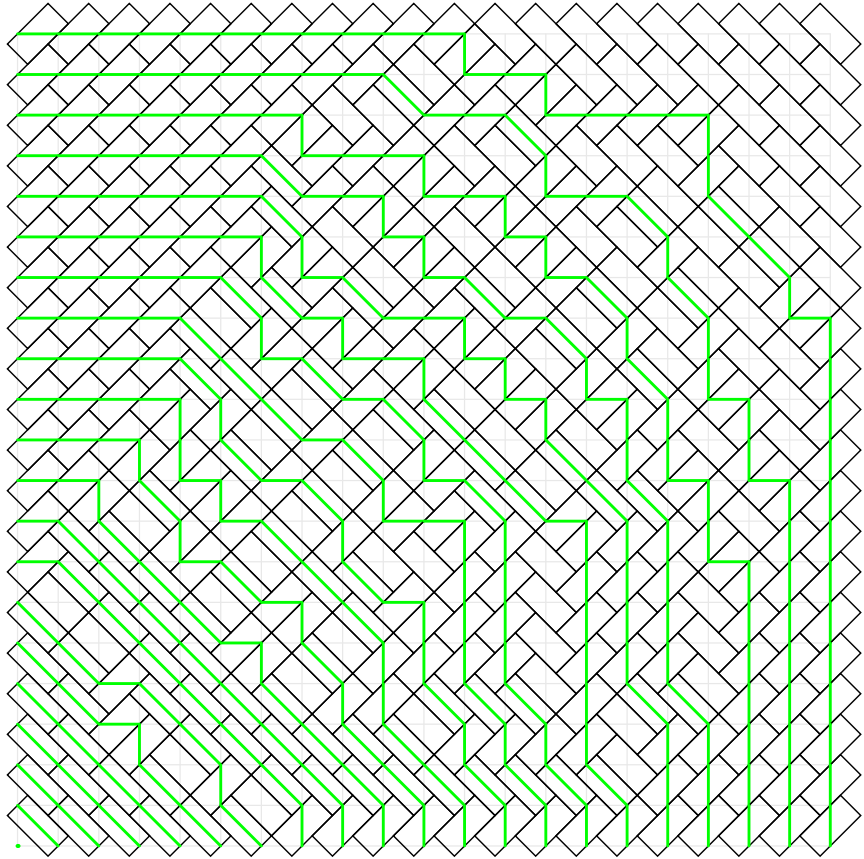}
  \caption{A domino tiling of the Aztec diamond of order~$20$, and (in green)
    the corresponding family of $21$ disjoint paths}
  \label{path-tiling}
\end{figure}
Using this bijection the Aztec diamond theorem was proved by Eu and Fu
\cite{EuFu}, by translating the enumeration of non-intersecting families of
lattice paths into the evaluation of certain Hankel matrices of Schr\"oder
numbers, which can be shown to give the proper power of~$2$ through a clever
interplay between algebraic and combinatorial viewpoints.

In this paper we propose another proof of the Aztec diamond theorem in terms
of non-intersecting families of lattice paths. We start by expressing the
number of such families (using the \LGV\ method) as a determinant (slightly
different from the one of~\cite{EuFu}), which can be evaluated by purely
algebraic manipulations. However we then also give a \emph{bijective} proof of
this enumeration, by giving a reversible procedure that constructs such
non-intersecting families from an array of $n(n+1)/2$ independent values taken
from $\{0,1\}$ (bits). Indeed we use these value to first construct a family
of $n+1$ (possibly intersecting) paths $P_i$ with $0\leq i\leq n$, where there
are $2^i$ possibilities for $P_i$; then we modify the family by a succession
of operations that may interchange steps among its paths, so as to ensure they
all become disjoint. These modifications are invertible step-by-step; to make
this precise we specify at each intermediate point of the transformation
precise conditions on the family that ensure that continuing both in the
forward direction and in the backward direction can be completed successfully.
As a consequence we obtain the descriptions of a number of collections of
intermediate families of paths, all equinumerous.

Ours is not strictly speaking the first bijective proof of the Aztec diamond
theorem. Indeed the fourth proof of the original paper, though not formulated
as a bijective proof, does give a ``domino-shuffling'' procedure (which is
more explicitly described in \cite[section~2]{arctic}), with the aid of which
one can build domino tilings of Aztec diamonds of increasing order, in a
manner that uses a net influx of $n(n+1)/2$ bits of external information (each
passage from a tiling of order~$\imn$ to order~$i$ uses $i$ bits), and such
that all these bits can be recovered from the final tiling produced. However,
in spite of some superficial similarities, the procedure we present is quite
different in nature. The main differences are that our procedure operates not
on tilings but on families of (possibly intersecting) paths, that it proceeds
in a regular forward progression rather than alternating deconstruction,
shuffling, and construction steps, and that this progression involves parts of
the final configuration successively attaining their final state rather than a
passage through complete configurations of increasing order. A more detailed
comparison will be given towards the end of our paper. Like domino shuffling,
our algorithm provides a simple and efficient means to produce large
``random'' examples of disjoint families of lattice paths as in
figure~\ref{circle} (or of domino tilings), which illustrate the ``arctic
circle'' phenomenon of \cite{arctic}.
\begin{figure}
  \includegraphics[width=0.9\textwidth]{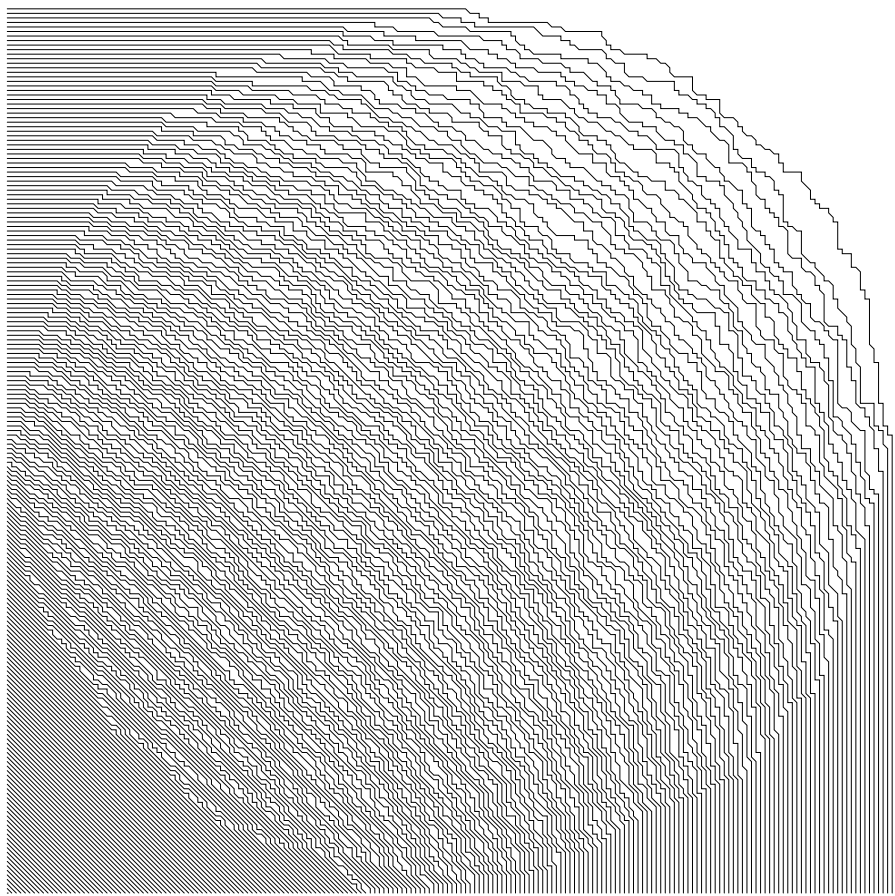}
  \caption{A random disjoint family of 196 paths}
  \label{circle}
\end{figure}

The domino tiling point of view in fact plays no role at all in our
construction; indeed we discovered the known connection with tilings of Aztec
diamonds only after the first author found the bijective proof as one of an
enumeration formula for families of lattice paths. In this paper we shall more
or less follow the route by which we approached the problem, leaving the
connection with Aztec diamonds aside until the final section. Henceforth $n$
will be the number of paths in a family, which is \emph{one more} than the
order of the corresponding Aztec diamond.

We give some definitions in section~2, and in section~3 enumerate disjoint
families of lattice paths using a determinant evaluation. In section~4 we give
some illustrations and considerations leading to an informal approach to our
algorithm, followed by a more formal statement and proof in section~5. Finally
we detail in section~6 the bijection between disjoint families of paths and
tilings of the Aztec diamond (a statement other than by pictorial example does
not seem to appear in the literature), and discuss some complementary matters.

\section{Definitions}

We shall consider paths through points in a square lattice whose basic steps
are either by a unit vector in one of two directions along the axes, or a
diagonal step by the sum of those two vectors. We shall call these
Schr\"oder-type paths. Concretely, since we shall want out paths to connect
points on the two borders of the positive quadrant
$\N\times\N\subset\Z\times\Z$, we take our basic steps to be by one of the
vectors $\pt0{+1}$, $\pt{-1}{+1}$ and $\pt{-1}0$, and the step will then
respectively be called horizontal, diagonal or vertical. This terminology
implies that we think the first index (or coordinate) varying vertically and
the second index varying horizontally, like in matrices. We shall frequently
refer to a set of vertically aligned points as a ``column''; in column~$k$ the
constant second index is equal to~$k$. However for visualisation it will be
slightly more convenient to have the first index increase upwards rather than
(as in matrices) downwards, so this is what we shall do. This amounts to using
the convention of Cartesian coordinates, but with the order of these
coordinates interchanged.

\begin{definition}
  A Schr\"oder-type path from $p$ to $q$, for points $p,q\in\Z\times\Z$, is a
  sequence $P=(p_0,p_1,\ldots,p_k)$ with $k\in\N$, $p_i\in\Z\times\Z$ for
  $0\leq i\leq k$, $p_0=p$, $p_k=q$, and
  $p_{i+1}-p_i\in\{\pt0{+1},\pt{-1}{+1},\pt{-1}0\}$ for $0\leq i<k$. The
  \emph{support} of $P$ is $\supp(P)=\{p_0,p_1,\ldots,p_k\}$.
\end{definition}

We denote by $a_{i,j}$ be the number of  Schr\"oder-type paths from
$\pt{i}0$ to $\pt0j$ (a number also known as the Delannoy number $D(i,j)$).
Then
\begin{equation}
  \label{aijdef}
  a_{i,0}=1=a_{0,j}\quad\text{and}\quad
  a_{i+1,j+1}=a_{i,j+1}+a_{i+1,j}+a_{i,j}\quad\text{for all $i,j\in\N$.}
\end{equation}

\begin{definition}
  The infinite matrix of these numbers is $A=(a_{i,j})_{i,j\in\N}$; its
  upper-left $n\times n$ sub-matrix is $A_{[n]}=(a_{i,j})_{0\leq i,j<n}$, for
  any $n\in\N$.
\end{definition}

Applying the \LGV\ method to the determinant of $A_{[n]}$ leads to the
following kind of families of $n$ Schr\"oder-type paths.

\begin{definition}
  If $\pi\in\Sym_n$ is a permutation of $[n]=\{0,1,\ldots,n-1\}$, then we
  shall call ``$\pi$-family'' any $n$-tuple $(P_0,P_1,\ldots,P_{n-1})$ where
  $P_i$ is a Schr\"oder-type path form $\pt{i}0$ to $\pt0{\pi_i}$ for
  $i\in[n]$. If $\pi$ is the identity permutation of~$[n]$ we shall call a
  $\pi$-family simply an ``$n$-family''. A $\pi$-family is called
  \emph{disjoint} if $\supp(P_0)$, $\supp(P_1)$, \dots\ and $\supp(P_{n-1})$
  are all disjoint.
\end{definition}

A $\pi$-family cannot be disjoint unless $\pi$ is the identity permutation. We
shall use general $\pi$-families only in the initial interpretation of
$\det(A_{[n]})$: after reducing its evaluation to counting disjoint families,
we shall only deal with $n$-families.

\begin{definition}
  A Schr\"oder $n$-family is an $n$-family $(P_0,\ldots,P_{n-1})$ with the
  property that for each $i$ the path $P_i$ does not pass to the side of the
  origin of the anti-diagonal line joining its initial and final points: in
  formula, for each point $(k,l)\in\supp(P_i)$ one has $k+l\geq i$.
\end{definition}

Paths in a Schr\"oder $n$-family are (similar to) actual Schr\"oder paths. A
simple induction argument shows that any disjoint $n$-family is a Schr\"oder
$n$-family. In formulating our bijective proof for the enumeration of disjoint
$n$-families, we shall employ only Schr\"oder $n$-families, but which are not
necessarily disjoint. One particular kind of Schr\"oder paths of interest is
the following.

\begin{definition}
  A Schr\"oder-type path $(p_0,p_1,\ldots,p_k)$ from $(i,0)$ to $(0,i)$ is
  called cliff-shaped if $p_i=(k-i,i)$, in other words if its first $i$ steps
  are either horizontal or diagonal, and any remaining steps are vertical. A
  cliff-shaped Schr\"oder $n$-family is an $n$-family whose paths are
  cliff-shaped Schr\"oder paths.
\end{definition}

Clearly any cliff-shaped Schr\"oder-type path is a Schr\"oder path; therefore
the qualification ``Schr\"oder'' in the final clause is automatic. For a
cliff-shaped Schr\"oder path from $(i,0)$ to $(0,i)$ the first $i$ steps can
be chosen independently to be horizontal or diagonal, after which the
remainder of the path is determined; therefore there are $2^i$ such paths, and
$2^{\binom{n}2}$ cliff-shaped Schr\"oder $n$-families.

\section{Enumeration of disjoint Schr\"oder $n$-families}

If we denote the set of $\pi$-families by $F(\pi)$ then we have
$\#F(\pi)=\prod_{i\in[n]}a_{i,\pi_i}$by definition of the numbers $a_{i,j}$,
and we can therefore evaluate
\begin{equation}
  \det(A_{[n]})
  =\sum_{\pi\in\Sym_n}\sg(\pi)\prod_{i\in[n]}a_{i,\pi_i}
  =\sum_{\pi\in\Sym_n}\sg(\pi)\#F(\pi).
\end{equation}
Now the \LGV\ method says we can replace the latter summation by its
contribution from disjoint families only, since all other contributions cancel
out. Indeed if a $\pi$-family $(P_0,\ldots,P_{n-1})$ has any pair of distinct
paths $P_i,P_j$ whose supports have non-empty intersection, one can modify
$P_i$ and $P_j$ by interchanging their parts beyond (in the obvious sense)
some point of that intersection to obtain a $\pi'$-family, with
$\pi'=\pi\circ(i~j)$ and hence $\sg(\pi')=-\sg(\pi)$, which therefore gives an
opposite contribution to the summation. It remains to make this cancellation
systematic, which can be done by fixing a rule that chooses for every
non-disjoint family a pair $\{i,j\}$ and a point of intersection of the
supports of $P_i$ and $P_j$, in such a way that the same choices will be
produced for the family obtained after modifying $P_i$ and $P_j$ by the
ensuing interchange; this will ensure one obtains a sign-reversing
\emph{involution} of the set of non-disjoint families. This rule can be chosen
in a multitude of ways (although it is not \emph{entirely} trivial to do so,
since the modification may change the set of candidate pairs~$\{i,j\}$ of
indices), and leave it to the reader to choose one.

Since a $\pi$-family can only be disjoint if $\pi$ is the identity
permutation, we find that $\det(A_{[n]})$ is equal to the number of disjoint
Schr\"oder $n$-families. On the other hand this determinant can be easily
evaluated recursively using algebraic manipulations. If $n>0$ and
$E_{[n]}=(\delta_{i,j}-\delta_{i+1,j})_{i,j\in[n]}$ is the upper unitriangular
$n\times n$ matrix with entries $-1$ directly above the diagonal and zeroes
elsewhere above the diagonal, then the product
$A'_{[n]}=\transp{E_{[n]}}A_{[n]}E_{[n]}=(a'_{i,j})_{i,j\in[n]}$ has entries
$a'_{i,j}$ that are $\delta_{i,j}$ if $i=0$ or $j=0$, and are otherwise given
by
\begin{equation}
  a'_{i,j}=a_{i,j}-a_{i,j-1}-a_{i-1,j}+a_{i-1,j-1}=2a_{i-1,j-1}
  \qquad\text{if $i,j>0$,}
\end{equation}
where the latter equality is a consequence of the recursion
relation~(\ref{aijdef}). This means that $A'_{[n]}$ can be written in
$(1,n-1)\times(1,n-1)$ block matrix form
\begin{equation}
  \transp{E_{[n]}}A_{[n]}E_{[n]}
   = \begin{pmatrix}1&0\\ 0&2A_{[n-1]}\\ \end{pmatrix}
  \qquad\text{if $n>0$,}
\end{equation}
from which, since $\det(E_{[n]})=1$, it follows that
$\det(A_{[n]})=2^{n-1}\det(A_{[n-1]})$ when $n>0$, so
\begin{equation}
  \det(A_{[n]})=2^{\binom{n}2}
  \qquad\text{for all $n\in\N$.}
\end{equation}
This proves

\begin{thm}\label{path thm}
  For $n\in\N$, the number of disjoint Schr\"oder $n$-families is
  $2^{\binom{n}2}$. \qed
\end{thm}

\section{Some illustrations, and informal approach to a bijection}

In this section we give some illustrations of the problem at hand, and some
considerations and examples that might help appreciate the bijective proof of
theorem~\ref{path thm} that we shall give. Impatient readers may skip to
the next section where this proof is given, and which is independent
of the current one. There the bijection will be formalised in the form of
pseudo-code; a computer program that implements this algorithm, and which was
used to prepare the illustrations in this paper, is available from the
website~\cite{site-Marc} of the second author.

We shall start by listing all $2^{\binom42}=64$ disjoint Schr\"oder
$4$-families, to give an impression of the variety these present. They are
displayed in figure~\ref{paths4}, ordered by increasing number of non-diagonal
steps from bottom(-left) to top(-right).

\begin{figure}[ht]
  \includegraphics[width=0.8\textwidth]{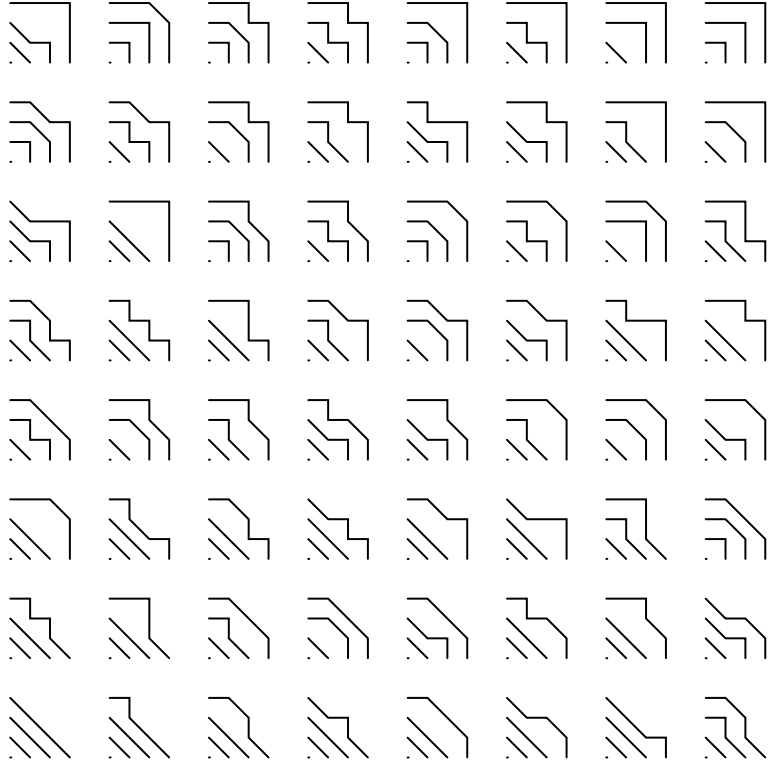}
  \caption{The collection of all disjoint Schr\"oder $4$-families}
  \label{paths4}
\end{figure}

A first fact that is apparent in this figure is that the number of horizontal
steps (which always equals the number of vertical steps), or by
complementation the number of diagonal steps, follows a (symmetric) binomial
distribution for $m=6=\binom42$ independent trials, as the frequencies are
$1,6,15,20,15,6,1$ respectively for $0,1,\ldots,6$ such steps. Even more
remarkably (if less obviously), the joint distribution of the number of
vertical steps in each of the four columns (vertical lines of the grid), which
we shall call the column counts, is the product of \emph{independent} binomial
distributions for $m=0,1,2,3$ respectively. The corresponding statements
remain true for the collection of all disjoint Schr\"oder $n$-families for any
$n\in\N$ (this will be obtained as a corollary of our bijective proof). By an
obvious symmetry one also has the corresponding statement for the joint
distribution of the number of horizontal steps on each of the four horizontal
lines of the grid (row counts), with $m$ increasing from bottom to top.

One also has a similar statement for joint distribution of what we shall call
inter-column counts, the number of horizontal steps connecting each pair of
successive columns; now $m$ decreases, from $n-1$ between the leftmost pair of
columns to $1$ between the rightmost pair. This statement can be seen to be
equivalent to the one about column counts, if one uses the duality illustrated
in figure~\ref{duality}; this is a bijection between the set of disjoint
Schr\"oder $n$-families and the set of such families transformed by a central
reflection sending the origin to $(n-\frac12,n-\frac12)$ (grid points are
mapped to centres of squares of the original grid).
\begin{figure}[hbt]
  \includegraphics[width=.5\textwidth]{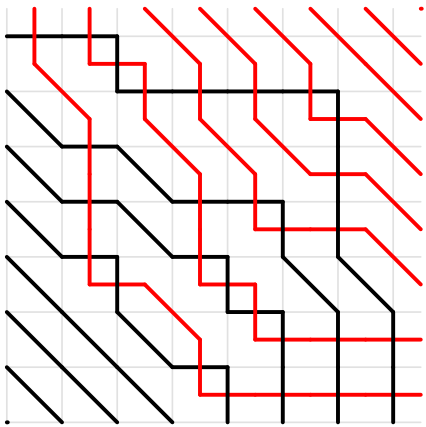}
  \caption{A disjoint $8$-family and (in red) its dual family}
  \label{duality}
\end{figure}
This correspondence is such that halfway on each horizontal or vertical step
of a disjoint $n$ family, the step crosses a vertical respectively horizontal
step of the dual family. By contrast to these facts, the joint distribution of
the number of horizontal (or equivalently vertical) steps in each of the
individual paths that make up a disjoint family does not satisfy any such
independence.

Given these observations, one may hope to find a bijection between disjoint
Schr\"oder $n$-families and triangular arrays of $\binom n2$ ``bits'' (values
in $\{0,1\}$) in such a way that, for a certain arrangement of the triangle
into columns of length~$i$ for $i=0,1,\ldots,n-1$, the sum of the bits in
column~$j$ will give the column count for column~$j$ of the corresponding
disjoint $n$-family.

Looking at just the $6$ paths with a single horizontal and vertical step, in
the bottom line of figure~\ref{paths4}, one sees that the point of
intersection of the lines contining these steps are all different, and form
the triangle of all grid points that are not visited by the paths of the
unique $4$-family with diagonal steps only (at the bottom left). This might
suggest placing the triangular array of bits on those grid points, in the hope
to find a bijection with disjoint $n$-families such that \emph{in addition}
the row sums of these bits give the row counts of the corresponding
$n$-family. This is easily seen to be impossible though, since the joint
distribution of the column counts and row counts of disjoint $n$-families is
different from the joint distribution of column sums and row sums in such
triangular array of bits. For instance for any $c\leq n$ there exist disjoint
$n$-families with $c$ horizontal and $c$ vertical steps, all of them
contributing to the \emph{same} column count respectively row count; when
$c\geq2$ the corresponding situation cannot occur for the column and row sums
of a triangular array of bits. On the other hand it may be checked in the
example that the joint distribution of column counts and inter-column counts
over all disjoint $4$-families is precisely that of column sums and row sums
in such triangular array of $\binom42=6$ bits. This suggests that in
formulating a bijection one should prefer to abandon the transposition
symmetry, and instead focus on (say) vertical alignment only. Indeed our
bijection will be such that column counts (of vertical steps) and inter-column
counts (of horizontal steps) can be immediately read off from the triangular
bit-array. However, the way these steps are distributed within their column
respectively inter-column space will not be so easy to read off.

The starting point of our bijection will then be to translate a triangular
array of $\binom n2$ bits into a cliff-shaped $n$-family, where line~$i$ of
the triangle (viewed in some appropriate direction) determines the
cliff-shaped path $P_i$. For such families of paths column counts and
inter-column counts are defined, just like for disjoint families; this time
one has the particular circumstance that only path $P_i$ contributes to the
column count for column~$i$. So each bit has two associated indices: that of a
path $P_i$ (which also gives the column to which it may contribute a vertical
step), and that of a an inter-column space, from column~$j$ to $j+1$, to which
it may contribute a horizontal step; the triangle runs through values $0\leq
j<i<n$. There does not seem to be a particularly suggestive way to view our
triangle as positioned in some specific way relative to the $n$-family; a
somewhat suggestive choice would be to take the set of (midpoints of)
horizontal steps in the ``no diagonal steps'' $n$-family (at the top right of
figure~\ref{paths4}).

Our main task will then be to find a systematic and reversible way to take any
cliff-shaped $n$-family and redistribute its horizontal and vertical steps
among the different paths, keeping each of these steps within its inter-column
space respectively within its column, so as to obtain a disjoint family. We
can give some heuristic arguments to explain the form that our algorithm will
take. For the redistribution of steps, the vertical steps will play a passive
role, since the fact that within column~$k$ they are originally all
concentrated in the path~$P_k$ makes that they initially carry very little
information. So we shall operate primarily on the initial parts of
cliff-shaped paths, which contain a mix of horizontal and diagonal steps;
whenever we move a horizontal step from one path to another (exchanging it
with a diagonal step), a corresponding vertical step will also be moved
between the paths so as to keep the ending point of that path unchanged.

An important aspect of our ``untangling'' procedure will be that it operates
essentially on parts of the paths that contain only horizontal and diagonal
steps. Since redistributing vertical steps in column~$k$ may move them from
path~$P_k$ into paths $P_i$ with $i>k$, it is practical to so treat columns
sequentially by \emph{decreasing} value of~$k$, and to leave column~$k$ as it
is once the vertical steps it contains are redistributed. In this way we avoid
having ``polluted'' paths with vertical steps in the columns under
consideration, and their parts beyond the column~$k$ where redistribution
currently takes place can be ignored by the procedure.

One more property of our procedure may be mentioned here, namely that if the
initial cliff-shaped $n$-family happens to be disjoint as well, we just leave
it as it is. Although this will involve a vanishingly small fraction of the
families as $n$ increases (notwithstanding the $26$ such cases out of $64$ for
$n=4$), the principle of acting only when clearly needed is an important guide
to understanding the procedure. This brings us to the following setting where
action may be required: we have two successive paths $P_i,P_{i+1}$ with $i\geq
k$, whose parts up to the point where they enter column~$k$ do not contain any
vertical steps, but which parts may intersect. At the point in time where we
start considering column~$k$ (vertical steps having been redistributed in all
columns beyond it), this situation occurs for $i=k$: since $P_k$ cannot have
been involved in any of the previous operations, it is in its initial state,
and could be any cliff-shaped path. In particular there no reason to suppose
anything about its position relative to $P_{k+1}$. And even though $P_{k+1}$
can have been operated upon, and therefore may be more likely to have certain
forms than others, it certainly \emph{can} also involve any sequence of
horizontal and diagonal steps before entering column~$k$. Indeed $P_{k+1}$
could also be in its initial state, as would happen if no action at all was
required before considering column~$k$, and as is certainly the case at the
very beginning, when for $k=n-2$ we consider the paths $P_{n-2},P_{n-1}$. So
apart from the absence of vertical steps we cannot assume anything about the
first $k$ steps of $P_i$ and~$P_{i+1}$. On the other hand we shall assume that
in column~$k$ only $P_i$ may have vertical steps initially, and also that
beyond column~$k$ the paths are already disjoint.

\begin{figure}[hbt]
  \includegraphics[width=.8\textwidth]{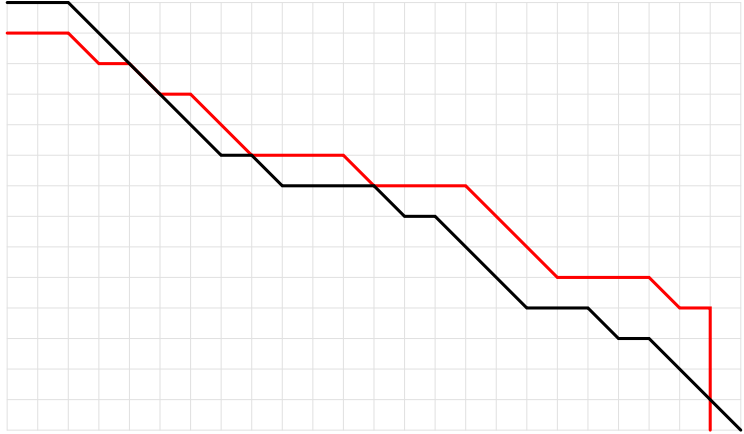}
  \caption{Initial parts of a pair of paths in need of untangling}
  \label{untangle-before}
\end{figure}
A typical situation is depicted in figure~\ref{untangle-before}; the red path
is~$P_i$ and the black one~$P_{i+1}$. The paths have been truncated to their
initial parts relevant to the task of untangling: path $P_{i+1}$ has no
vertical steps in column $k=23$ and passes to column~$k+1$, while path $P_i$
does have at least the vertical steps in column~$k$ shown. It may be that
$P_i$ continues further downwards (as it will when $i=k$, since then $P_i$ is
cliff-shaped), or it may pass to column~$k+1$ as well; but if it does, it must
do while staying below~$P_{i+1}$.

Since the paths depicted first meet in column~$4$, the principle to act only
when needed suggests leaving everything up to column~$3$ intact. We might then
avoid the collision in column~$4$ either by taking a diagonal step in~$P_i$ or
by taking a horizontal step in~$P_{i+1}$, but if we want to keep the number of
horizontal steps unchanged, and more precisely the number of horizontal steps
from column~$3$ to column~$4$, the only (easy) way to achieve this is by
making \emph{both} these changes. As this transfers a horizontal step from
$P_i$ to $P_{i+1}$, we shall also need to transfer a vertical step, in
column~$k$. As we shall see below, the latter transfer combined with the
initial absence of vertical steps in~$P_{i+1}$ is a key point in being able to
reverse the modification(s) made, as it serves as witness for the effort that
was required to make the pair of paths disjoint.

Having ``switched step directions'' between columns $3$~and~$4$, the
remainders of $P_i$ and~$P_{i+1}$ are shifted down respectively up by one
unit. It might seem that the next (and only) remaining problem that needs
resolving occurs in the passage to column~$15$, where the original path $P_i$
rises \emph{two} units above~$P_{i+1}$ for the first time, so that the
mentioned remainders meet in spite of the shifts. However, while switching
step directions in the passages to columns $4$~and~$15$ only (and moving two
vertical steps to~$P_{i+1}$) would succeed in making the paths disjoint, the
result leaves insufficient information to reconstruct the set of steps that
were adjusted, and hence the initial paths. The modified steps cause the new
paths to move apart at a point where they are as close together as they may,
but so do the passages to columns $8$ and $12$ (in the modified paths), with
nothing to distinguish these cases. Therefore, we shall instead switch
directions \emph{every} time that the height of $P_i$ above $P_{i+1}$ first
reaches a new nonnegative value, which in the example happens for the values
$0,1,2,3$ when passing respectively to columns $4$, $6$, $15$ and~$23$. The
result of those four interchanges, and of moving $4$ vertical steps from $P_i$
to $P_{i+1}$ in column~$23$, is shown in figure~\ref{untangle-after}.
\begin{figure}[hbt]
  \includegraphics[width=.8\textwidth]{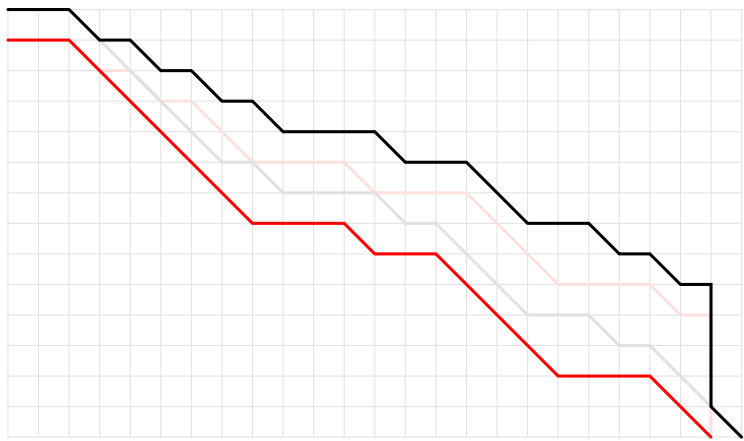}
  \caption{Initial parts of the pair of paths untangled}
  \label{untangle-after}
\end{figure}

One can view this transformation in terms of a single path $\Delta$ of a new
kind, defined by the ``difference'' of $P_{i+1}$ and $P_i$: one that makes a
down-step whenever between two columns $P_{i+1}$ has a diagonal step and $P_i$
a horizontal one, an up-step when $P_{i+1}$ has a horizontal step and $P_i$ a
diagonal one, and a neutral step when $P_{i+1}$ and $P_i$ have the same type
of step (in $\Delta$ the two kinds of neutral steps are distinguished, so that
no information is lost). Then $P_i$ and~$P_{i+1}$ are disjoint if and only if
the maximal depth~$d$ beneath its starting level to which $\Delta$ descends
is~$0$, and we have described a procedure to transform any $\Delta$ into a
path with $d=0$, by reversing all down-steps that lead to a new left-to-right
minimum. The procedure is well known in this setting, and in various
equivalent guises; see for instance \cite{lattice walks} and references
therein. The mapping it defines has the important property of becoming
injective when restricted to paths with a given initial value of~$d$. This can
be seen by viewing the transformation as obtained by iterating as long as
possible the operation of reversing the first down-step that leads to the
\emph{globally} minimal level (which is initially~$d$); this iteration
produces the reversals from right to left. Each such operation is invertible
by reversing the last up-step starting at the globally minimal level, so given
$d$ one can undo the entire transformation by repeating this inverse operation
$d$ times.

The procedure described allows making a pair of successive paths $P_i,P_{i+1}$
disjoint up to column~$k$, and is reversible provided that $P_{i+1}$ initially
had no vertical steps in that column. Assuming that the paths
$P_{k+1},\ldots,P_{n-1}$ have previously been made disjoint, we can use this
procedure to make $P_k$ disjoint from~$P_{k+1}$. But since this in general
involves moving parts of both paths away from each other, it may cause
$P_{k+1}$ to intersect $P_{k+2}$ even though they were disjoint before. In
fact one could not expect being able to make $P_k,\ldots,P_{n-1}$ disjoint so
easily: one needs to potentially introduce vertical steps in column~$k$ for
\emph{all} these paths. After all, once this disjointness is obtained, further
transformations will no longer change column $k$, and for each $i\geq k$ there
certainly exist disjoint $n$-families in which $P_i$ has one or more vertical
steps in column~$k$.

An obvious idea is then to continue applying the untangling procedure as long
as there are pairs of adjacent paths that intersect. But unless this process
proceeds in a very orderly fashion, it will be problematic to invert, and
could even fail to terminate. Fortunately it turns out that the process is
indeed very orderly: if after untangling $P_i$ and~$P_{i+1}$ we need to
untangle $P_{i+1}$ and~$P_{i+2}$, then this may cause $P_{i+1}$ to ``bounce
back'' towards $P_i$, but when this happens the extra space that their initial
untangling had produced between $P_i$ and~$P_{i+1}$ is always sufficient to
absorb the displacement of~$P_{i+1}$, thus avoiding any new intersection
between them. Given this state of affairs, which we shall prove in the next
section, a single sweep of untangling of paths by increasing value of~$i$,
starting at $i=k$, will suffice. The sweep will end when no new intersections
are produced, which at the very last is bound to happen after untangling the
final paths $P_{n-2}$ and~$P_{n-1}$, if one ever gets to that point.

The succession of intermediate paths families during such a sweep of
executions of the untangling procedure for increasing values of $i$ is
illustrated in figure~\ref{sweep}, with the path $P_i$ for the next such
execution in red. In the very last such execution, the paths are found to be
disjoint already and nothing is changed.
\begin{figure}[hbt]
  \includegraphics[width=\textwidth]{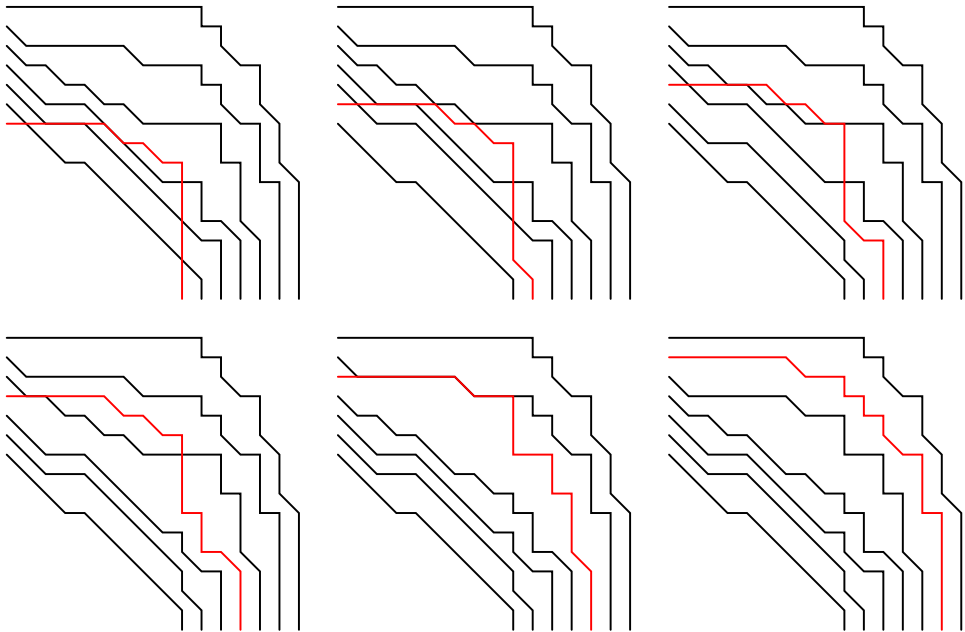}
  \caption{One sweep of untangling until disjointness is obtained}
  \label{sweep}
\end{figure}

To put everything together, it remains to start with a cliff-shaped $n$-family
determined by a triangular array of $\binom n2$ bits, and apply the above
``sweeps'' distributing vertical steps in column~$k$ among the paths, for
$k=n-2,\ldots,2,1$. This process is illustrated in figure~\ref{phases},
showing the transformation of a cliff-shaped $49$-family into a disjoint
$49$-family in several stages, including the initial and final ones. To avoid
distraction, not yet treated cliff-shaped paths, which intersect each other
and the already ``combed'' ones, are in light blue.
\begin{figure}
  \includegraphics[width=\textwidth]{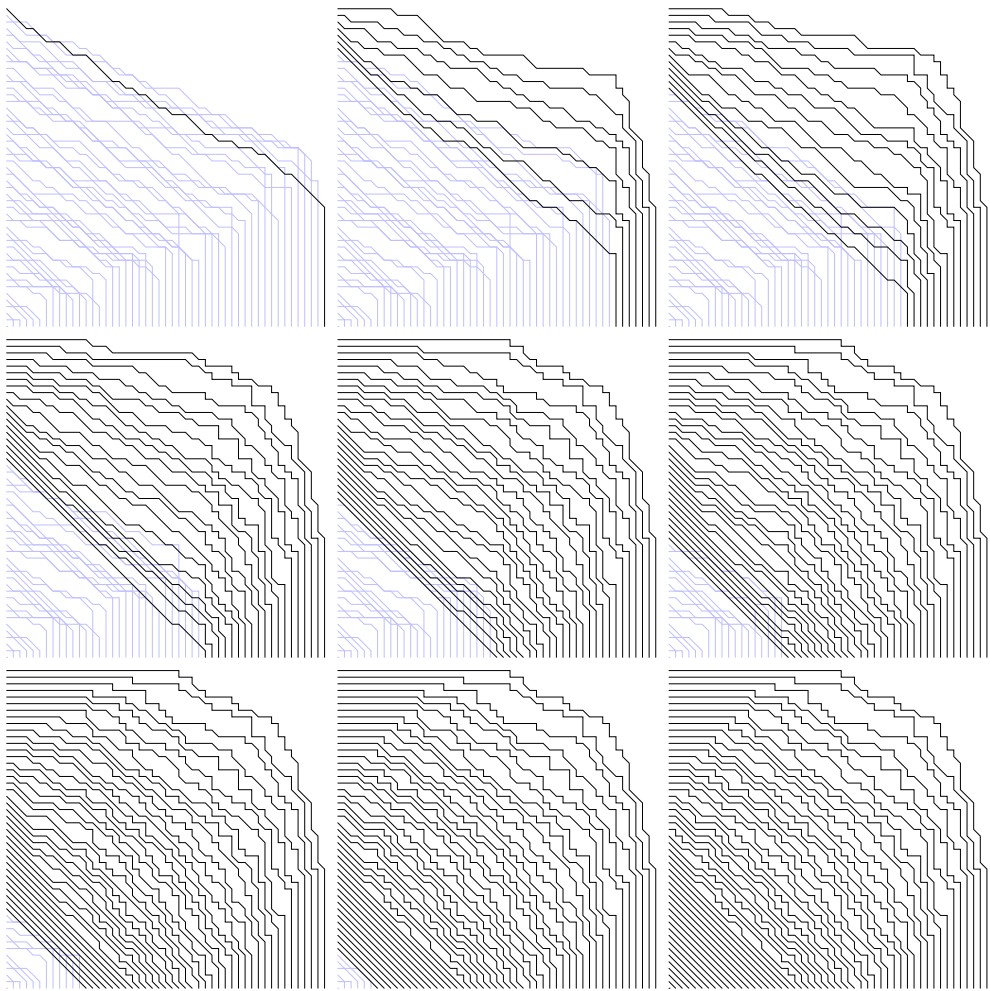}
  \caption{Several intermediate phases of combing $49$ paths}
  \label{phases}
\end{figure}

\section{A bijective proof}

We shall now formulate a bijective proof of theorem~\ref{path thm}, by giving
an algorithmically defined bijection between the set of disjoint Schr\"oder
$n$-families and the set of cliff-shaped Schr\"oder $n$-families, the latter
set having the number of elements mentioned in the theorem. We shall focus
first on the direction from cliff-shaped to disjoint $n$-families, where the
goal is to remove intersections between pairs of paths (this is what
``combing'' in our title refers to). However our claim that the map so defined
is a bijection depends the existence of an inverse transformation defined for
any disjoint $n$-families, and which defines an inverse mapping; in this
direction the ``goal'' is to move, for all~$k$, all vertical steps in
column~$k$ towards path~$P_k$, where they will appear at the end, so that the
paths become cliff-shaped. Whenever one algorithm transforms a family of one
type into another, the other algorithm applied to the family produced will
realise a step-by-step inverse of the initial transformation.

Our basic operations, forward and backward, operate on a pair of successive
paths $(P_i,P_{i+1})$ in a Schr\"oder $n$-family (the others are ignored), and
depend on an additional parameter $k\leq i$. These paths should have no
vertical steps in columns $j<k$; the operations will not introduce such steps
either. Moreover they leave each of the paths unchanged beyond column~$k$, so
the only way vertical steps play a role is by a possible transfer between the
paths of vertical steps in column $k$. The forward operation defines a
bijection from the pairs of such paths for which $P_{i+1}$ does not have any
vertical steps in column~$k$ while $P_i$ has enough of such steps in a sense
to be made precise, to the pairs of such paths with disjoint supports. In the
context where we shall apply the operations, these conditions will be
satisfied, and $P_i,P_{i+1}$ will also be disjoint beyond column~$k$.

In what follows the following assumptions are tacitly made: all paths will be
assumed to without vertical steps in columns $j<k$, the paths $P_i$ and $P'_i$
are Schr\"oder paths from $(i,0)$ to $(0,i)$, and the paths $P_{i+1}$ and
$P'_{i+1}$ are Schr\"oder paths from $(\ip,0)$ to $(0,i+1)$. The absence of
vertical steps allows the parts of such paths up to column~$k$ to be viewed as
graphs of functions: for $\delta\in\{0,1\}$ and $0\leq j\leq k$, let
$h_\delta(j)$ be the greatest (and unique, unless $j=k$) value~$v$ with
$(v,j)\in\supp(P_{i+\delta})$. These functions $h_0$ and $h_1$ are weakly
decreasing, and their value decreases by at most~$1$ at each step.

For defining the forward operation (and so with the mentioned assumptions
on $P_i,P_{i+1}$), put
\begin{equation}
  \label{djdef}
  d_j=\max\{\,h_0(j')+1-h_1(j')\mid 0\leq j'\leq j\,\}
  \qquad\text{for $0\leq j\leq k$}.
\end{equation}
The sequence $(d_0,d_1,\ldots,d_k)$ is weakly increasing, and by at most~$1$
at each step; it starts with $d_0=0$. One will (still) have $d_k=0$ if and
only if the paths $P_i$ and $P_{i+1}$ have disjoint supports up to column~$k$.
Now define $h'_0,h'_1$ by
\begin{equation}
  \label{fprimedef}
  h'_0(j)=h_0(j)-d_j \quad\text{and}\quad h'_1(j)=h_1(j)+d_j
  \qquad\text{for $0\leq j\leq k$.}
\end{equation}
There is at most one pair of paths $(P'_i,P'_{i+1})$, unchanged from their
final points in column~$k$ on with respect to $(P_i,P_{i+1})$, that gives
rise to $(h'_0,h'_1)$ in the same way as $(P_i,P_{i+1})$ gives
rise to $(h_0,h_1)$. Our operation is defined only when such $(P'_i,P'_{i+1})$
exists, and then replaces $P_i$ by $P'_i$ and $P_{i+1}$ by~$P'_{i+1}$.

For any $j<k$, the steps in $P'_i,P'_{i+1}$ from column $j$ to column $j+1$
will be of the same type as the corresponding steps in $P_i,P_{i+1}$
respectively, unless $d_j<d_{j+1}$. By (\ref{djdef}), the latter case occurs
only in situations where $h_0(j)=h_0(j+1)$ and $h_1(j)>h_1(j+1)$, in other
words when the step from column $j$ to column $j+1$ is horizontal in $P_i$ and
diagonal in $P_{i+1}$. When indeed $d_j<d_{j+1}$, these directions are
interchanged in $P'_i,P'_{i+1}$: the step from column $j$ to column $j+1$ is
diagonal in $P'_i$ and horizontal in~$P'_{i+1}$. This situation arises $d_k$
times in all. As a result, $P'_{i+1}$ has $(h_1(k)+d_k,k)$ as first point
in column~$k$, after which it has $d_k$ vertical steps to reach the point
$(h_1(k),k)$ where the original path $P_{i+1}$ enters column~$k$.

The path $P'_i$ on the other hand will have $d_k$ vertical steps \emph{less}
in column~$k$ than~$P_i$ has. The (unique) condition for the existence of
$(P'_i,P'_{i+1})$ then is that $P_i$ has at least that many such steps to
begin with. So we can detail the requirement alluded to above that $P_i$ have
enough vertical steps in column~$k$: we must assume that it has at least~$d_k$
such steps, as defined in~(\ref{djdef}). An equivalent, maybe more natural,
way of stating this requirement is that if we would modify~$P_i$ by removing
all its vertical steps from column~$k$ and insert them into column~$0$ instead
(shifting all intermediate steps), then the resulting (Schr\"oder-type but maybe not Schr\"oder) path would have its support disjoint
from that of $P_{i+1}$.

It is clear that $h'_0(j)<h'_1(j)$ for all $0\leq j\leq k$, since
\begin{equation}
  \label{fprimeineq}
  h'_1(j)-h'_0(j)-1=h_1(j)-h_0(j)-1+2d_j  %=(d_j-(h_0(j)+1-h_1(j)))+d_j
  \geq d_j,
\end{equation}
and $d_j\geq0$; moreover for $j=k$ one gets that $h'_0(k)<h'_1(k)-d_k=h_1(k)$,
which is the first coordinate of the point where $P_{i+1}$ enters into
column~$k$, and by the assumption that $P_{i+1}$ has no vertical steps in
column~$k$, this point is also the last one of~$P'_{i+1}$ in that column. This
shows that the supports of $P'_i$~and~$P'_{i+1}$ are disjoint up to column~$k$
inclusive. In fact this inequality shows that these paths leave at least $d_j$
empty places between them in any column~$j<k$, so whenever $d_j$ increases
with~$j$, the modified paths are forced to remain further and further apart.
Thus an increase $d_j<d_{j+1}$ not only implies one has a diagonal step
in~$P'_i$ and a horizontal step in~$P'_{i+1}$ between columns $j$~and~$j+1$,
but also that the paths then continue to leave this increased number~$d_{j+1}$
of spaces (or more) between them, until they enter column~$k$.

The backward operation uses this property to detect the points of increase
of~$d_j$ from the shape of the paths $P'_i,P'_{i+1}$ alone (so that it can
then reconstruct $(d_0,\ldots,d_k)$), but needs to distinguish this situation
from one where the \emph{original} difference $h_1(j)-h_0(j)$ increases
at~$j+1$ without ever falling back subsequently. But in the latter case one
has $d_k=d_j<h'_1(j)-h'_0(j)$ (the equality follows from ``not falling back'',
and the inequality from (\ref{fprimeineq})), whereas in the case $d_j<d_{j+1}$
one has instead $d_k\geq d_{j+1}=d_j+1=h'_1(j)-h'_0(j)$ (the final equality
holds because the maximum in (\ref{djdef}) must be attained for $j'=j$).
Therefore one can tell the two cases apart provided that $d_k$ is known. But
that is the case: $P_{i+1}$ has no vertical steps in column~$k$, so one can
read off $d_k$ as the number of vertical steps of $P'_{i+1}$ in column~$k$.

So we can now formulate the backward operation, which can be applied to a pair
of paths $(P'_i,P'_{i+1})$ with supports disjoint up to column~$k$ inclusive.
We start by defining functions $h'_0,h'_1$ in terms of respectively
$P'_i,P'_{i+1}$, as before, and in addition let $d$ be the number of vertical
steps of $P'_{i+1}$ in column~$k$; then define the sequence
$(d_0,d_1,\ldots,d_k)$ by
\begin{equation}
  \label{djdefx}
  d_j=\min(\{d\}\cup\{\,h'_1(j')-h'_0(j')-1\mid j\leq j'\leq k\,\})
  \qquad\text{for $0\leq j\leq k$}.
\end{equation}
We then find $h_0,h_1$ by using equation (\ref{fprimedef})
in the opposite direction:
\begin{equation}
  \label{fdef}
  h_0(j)=h'_0(j)+d_j \quad\text{and}\quad h_1(j)=h'_1(j)-d_j
  \qquad\text{for $0\leq j\leq k$,}
\end{equation}
and finally take $(P_i,P_{i+1})$ to be the unique pair of paths, unchanged
with respect to $(P'_i,P'_{i+1})$ from their final points in column~$k$
onwards, giving rise to $(h_0,h_1)$.

Several easy verifications suffice to see that this backward operation is well
defined. The sequence $(d_0,d_1,\ldots,d_k)$ is weakly increasing by at most
one at each step, and satisfies $d_k=d$ (since the supports of $P'_i$ and
$P'_{i+1}$ are disjoint in column~$k$) and $d_0=0$ (since the disjointness of
the supports of $P'_i$ and $P'_{i+1}$ in column~$j$ gives
$h'_1(j)-h'_0(j)-1\geq0$, while $h'_1(0)-h'_0(0)-1=0$). All $d$ vertical steps
in column~$k$ of $P'_{i+1}$ are absent from $P_{i+1}$ but transferred
to~$P_i$, and the steps in $P_i$ and~$P_{i+1}$ from column~$j$ to $j+1$ stay
of the same kind as respectively in $P_i$ and~$P_{i+1}$ when $d_j=d_{j+1}$,
while the steps interchange directions when $d_j<d_{j+1}$;
this establishes the existence of $(P_i,P_{i+1})$.

When the pair $(P'_i,P'_{i+1})$ to which the backward operation is applied was
itself obtained by the forward operation from $(P_i,P_{i+1})$, it can be
checked that in the backward operation $d=d_k$, and that the sequence
$(d_0,\ldots,d_k)$ is the same as it was in the forward operation (the
condition causing $d_j<d_{j+1}$ in the backward operation is equivalent to the
one for which we argued that it it characterises $d_j<d_{j+1}$ in the forward
operation); in this case the pair obtained in the backward operation is
therefore the original pair~$(P_i,P_{i+1})$. Conversely, if the backward
operation is applied to any applicable pair $(P'_i,P'_{i+1})$, then the
forward operation can be applied to the resulting pair~$(P_i,P_{i+1})$, and it
will reconstruct $(P'_i,P'_{i+1})$. Again this follows by showing that the
forward operation reproduces the same sequence $(d_0,\ldots,d_k)$ as the
backward operation, as follows. For a maximal interval of consecutive
indices~$j$ for which during the backward operation $d_j$ has a constant
value, say~$c$, one has the relation $h_0(j)+1-h_1(j)=2c-(h'_1(j)-h'_0(j)-1)$
throughout. Also the maximal value of this expression is attained for the
minimum such~$j$ (as well as for the maximum such~$j$, provided it is less
than~$k$). Therefore during the forward operation, the value of $d_j$ from
(\ref{djdef}) will be constant on such intervals as well. On the other hand,
when $d_j<d_{j+1}$ during the backward operation, one has
$h_0(j)+1-h_1(j)=h_0(j+1)-h_1(j+1)$, and together with the constancy result we
just gave this shows that $d_j<d_{j+1}$ during the forward operation as well,
and therefore that $(d_0,\ldots,d_k)$ is reconstructed identically.

Let us resume the description of these basic operations as somewhat more
formalised computational procedures. To that end we need a concrete
representation of the $n$-families of paths operated upon. We choose a
representation that facilitates handling paths with a varying number of steps,
and allows making evident the simple structure of our operations. An
$n$-family of paths is encoded by a pair of lower triangular matrices $(B,D)$
indexed by $[n]\times[n]$ (recall that $[n]=\{0,1,\ldots,n-1\}$). The matrix
$B$ is strictly lower triangular with entries in~$\{0,1\}$, while $D$ is
weakly lower triangular with entries in~$\N$. the entry $B_{i,j}$ indicates
the direction of the step in $P_i$ between column $j$~and~$j+1$ (a value~$0$
for horizontal, or $1$ for diagonal), and the entry $D_{i,j}$ counts the
number of vertical steps of $P_i$ in column~$j$. A cliff-shaped $n$-family is
determined by $B$ alone, and the forward ``combing'' algorithm will gradually
compute $D$ for the corresponding disjoint $n$-family from it while updating
$B$ to match it. The reverse ``uncombing'' algorithm takes a disjoint
$n$-family encoded by $B,D$ and computes $B$ for the corresponding
cliff-shaped $n$-family from it.

The forward basic operation, which will make paths $P_i,P_{i+1}$ disjoint up
to column~$k\leq i$ inclusive, assumes $D_{i,k}$ is already determined, and at
the end of its execution transfers part of its value to $D_{i+1,k}$ (taken to
be $0$ initially). Its description in procedure~\ref{op-forward} uses local
variables $\textit{cur}\in\Z$ recording the current value of
$h_0(j)+1-h_1(j)$, and $d\in\N$ recording the maximum of \textit{cur} so far.
In this pseudo-code `$\gets$' denotes assignment of a new value, and we write
indices in square brackets to remind that this describes individually
assignable entries.

\begin{algorithm}
  \caption{Forward operation on paths $i,i+1$ up to column $k$ inclusive}
  \label{op-forward}
  \begin{algorithmic}
  \newcommand\cur{\textit{cur}}
    \STATE $\disj{i}:$
    \STATE $\cur\gets0$,\quad $d\gets0$
    \FOR {$j \FROM 0 \UPTO k-1$}
    \STATE $\cur \gets \cur + B[i+1,j] - B[i,j]$
    \IF {$cur>d$}
    \STATE $d\gets \cur$
    \STATE $B[i,j]\gets1$,\quad $B[i+1,j]\gets0$
    \ \COMMENT {interchange directions of steps}
    \ENDIF \ENDFOR
    \STATE $D[i,k] \gets D[i,k] - d$,\quad  $D[i+1,k] \gets d$
    \ \COMMENT{transfer $d$ vertical steps to $P_{i+1}$}
  \end{algorithmic}
\end{algorithm}

The backward operation in procedure~\ref{op-backward} retraces the steps of
procedure~\ref{op-forward} using the same local variables \textit{cur} and
$d$. While the sequence of values of $d$ retraces those in
procedure~\ref{op-forward} in reverse order, the values of \textit{cur} are
different: they record the current value of $h'_1(j)-h'_0(j)-1$ for the
functions $h'_0,h'_1$ corresponding to the disjoint paths described by the
initial values for procedure~\ref{op-backward}; in particular
$\textit{cur}\geq0$ throughout the execution. In order to set \textit{cur}
correctly, it assumes that the values $h'_0(k)$ and $h'_1(k)$, where paths
$P_i$ and $P_{i+1}$ respectively enter column~$k$ (which values are not
available directly in our encoding), have been stored beforehand as elements
$h_i$, $h_{i+1}$ of an auxiliary array; these values are updated to reflect
the effect of the operation.

\begin{algorithm}
  \caption{Backward  operation on paths $i,i+1$ up to column $k$ inclusive}
  \label{op-backward}
  \begin{algorithmic}
  \newcommand\cur{\textit{cur}}
    \STATE $\clify{i}:$
    \STATE $d \gets D[i+1,k]$,\quad $\cur \gets h[i+1] - h[i]-1$
    \ \COMMENT {$0\leq d\leq\cur$}
    \STATE $D[i+1,k]\gets0$,\quad  $D[i,k]\gets D[i,k]+d$
    \ \COMMENT{transfer $d$ vertical steps to $P_i$}
    \STATE $h[i+1] \gets h[i+1]-d$,\quad  $h[i] \gets h[i]+d$
    \COMMENT {adapt entry point into column}
    \FOR {$j \FROM k-1 \DOWNTO 0$}
    \STATE $cur \gets cur + B[i+1,j] - B[i,j]$
    \IF {$cur<d$}
    \STATE $d\gets cur$
    \STATE $B[i,j]\gets0,\quad B[i+1,j]\gets1$
    \ENDIF \ENDFOR
  \end{algorithmic}
\end{algorithm}

We can now formulate somewhat more formally what was proved above about the
forward and backward operations, as statement about the given procedures. For
conciseness we denote by $\Pathfam(n)$ the set of pairs of matrices $(B,D)$
where $B$ is strictly lower triangular $[n]\times[n]$ matrix with entries in
$\{0,1\}$, while $D$ is weakly lower triangular $[n]\times[n]$ matrix with
entries in~$\N$.

The procedures obviously only inspect and alter a small part of these
matrices, but there is no need to make explicit mention of that fact.
The fact that, as we proceed along the path~$P_i$, the level \emph{decreases}
by the values of $B_{i,j}$ and $D_{i,j}$ encountered, has as consequence that
the inequalities below are in the opposite direction as the corresponding
comparison of the levels of two paths. Also we have chosen to leave out
the respective initial levels $i$ and $i+1$ of the paths $P_i$~and~$P_{i+1}$
from the expressions, so when interpreting the inequalities as comparisons of
levels, one should take into account the difference in offset.

\begin{prop} \label{basic bijection}
  For $0\leq k\leq i<n-1$, procedure \ref{op-forward} defines a bijection, and
  procedure \ref{op-backward} defines the inverse bijection, between on one
  hand the set of pairs $(B,D)\in\Pathfam(n)$ satisfying
  \begin{align*}
    D_{i+1,k}&=0,\quad\text{and}\\
    \sum\nolimits_{j'=0}^{j-1}B_{i+1,j'}& \leq D_{i,k}+
    \sum\nolimits_{j'=0}^{j-1}B_{i,j'}, \qquad\text{for $0\leq j\leq k$,}
    \end{align*}
  and on the other hand the set of pairs
  $(B,D)\in\Pathfam(n)$ satisfying
  \begin{align*}
    \sum\nolimits_{j'=0}^{j-1}B_{i+1,j'}&\leq
    \sum\nolimits_{j'=0}^{j-1}B_{i,j'},\qquad\text{for $0\leq j<k$, and}\\
    \sum\nolimits_{j=0}^{k-1}B_{i+1,j}+D_{i+1,k}&\leq
             \sum\nolimits_{j=0}^{k-1}B_{i,j}.
  \end{align*}
  The relations $h_{i'} = i'-\sum_{j=0}^{k-1}B_{i',j}$ for $i'=i,i+1$ are assumed
  to hold initially in procedure \ref{op-backward}, and continue to hold after
  its execution. \qed
\end{prop}

We now build an algorithmic bijection corresponding to theorem~\ref{path thm}
by repeated application of basic operations. The iteration itself is
straightforward, although a bit of work will remain to show that the goal is
attained. For a given value of~$k$, we shall start calling $\disj{k}$ to make
$P_k$ and~$P_{k+1}$ disjoint (recall that $P_k$ does not extend beyond
column~$k$), then $\disj\kp$ to make $P_{k+1}$ and~$P_{k+2}$ disjoint up to
column~$k$, and so forth up to $\disj{n-2}$ to make the last two paths
$P_{n-1}$ and~$P_{n-2}$ disjoint up to column~$k$. We shall show that the
disjointness obtained in a step is not lost in the following step, so this
iteration will result in paths $P_k,\ldots,P_{n-1}$ being disjoint up to
column~$k$. Placing the iteration within another iteration, in which $k$
decreases from $n-2$ to $0$, we ensure that all paths that extend beyond
column $k$ are already disjoint when this inner iteration starts. Since the
parts beyond column $k$ are unaffected by it, the inner iteration will in fact
achieve that $P_k,\ldots,P_{n-1}$ are entirely disjoint, and at the end of the
outer iteration the whole $n$-family will be disjoint. Note that in general
applying $\disj{i}$ \emph{will destroy} the disjointness of $P_{i+1}$ and
$P_{i+2}$ up to column~$k$, which explains why the inner iteration is needed.

Since $\disj{i}$ will set the value of $D_{i+1,k}$ for use in the subsequent
$\disj{\ip}$, all that remains to do is to ensure that $D_{k,k}$ is set
correctly before the inner iteration at~$k$ starts; this is easy since the
number of final vertical steps in the cliff-shaped path $P_k$ is equal to its
number of horizontal steps. We obtain the combing algorithm described in
procedure~\ref{combing}.

\begin{algorithm}
  \caption{Combing algorithm from cliff-shaped to disjoint $n$-families}
  \label{combing}
  \begin{algorithmic}
    \FOR {$k \FROM n-1 \DOWNTO 0$}
    \STATE $D[k,k]\gets k-\sum_{0\leq j<k}B[k,j]$
    \ \COMMENT {initialise diagonal entry}
    \FOR {$i \FROM k \UPTO n-2$}
    \STATE $\disj{i}$
    \ENDFOR
    \ENDFOR
  \end{algorithmic}
\end{algorithm}

A first verification to be made is that the condition of
proposition~\ref{basic bijection} is satisfied whenever $\disj{i}$ is invoked.
This is clear initially when $i=k$, since the initialisation of $D_{k,k}$
gives that $D_{k,k}+\sum_{j'=0}^{j-1}B_{i,j'}=k-\sum_{j'=j}^{k-1}B_{i,j'}\geq
j$. To prove that the inequality is satisfied when $i>k$, we need the
hypothesis that the paths $P_i$ and $P_{i+1}$ were disjoint just before
$\disj{i-1}$ was executed. This means that one has
$\sum_{j'=0}^{j-1}B_{i+1,j'}\leq\sum_{j'=0}^{j-1}B_{i,j'}$ for $0\leq j\leq
k$ at the start of $\disj{i-1}$. If $d=D_{i,k}$ is the final value obtained by
this variable during that execution, then for any such~$j$ the value of
$\sum_{j'=0}^{j-1}B_{i,j'}$ is decreased by at most~$d$ by the procedure, and
since the values $B_{i+1,j'}$ are unaffected, one obtains
$\sum_{j'=0}^{j-1}B_{i+1,j'}\leq\sum_{j'=0}^{j-1}B_{i,j'}+d$ at the end of
$\disj\imn$, and therefore at the beginning of $\disj{i}$; this is the
condition required.

A reverse (uncombing) algorithm is also easy to formulate. Here both
$B$~and~$D$ have well defined values initially, and the only initialisation
required is that of the vector~$h$, which should give the levels at which
paths $P_i$ and $P_{i+1}$ enter column~$k$ at the point where
$\clify{i}$ is invoked, as mentioned in proposition~\ref{basic
  bijection}. Since procedure~\ref{op-backward} takes care of updating the
vector~$h$ according to the changes to~$B$ it produces, these initialisations
are easily integrated into the uncombing algorithm, which only needs to take
care of the passage from column $k-1$~to~$k$. We obtain the algorithm
described in procedure~\ref{uncombing}.

\begin{algorithm}
  \caption{Uncombing algorithm from disjoint to cliff-shaped $n$-families}
  \label{uncombing}
  \begin{algorithmic}
    \FOR {$k \FROM 0 \UPTO n-1$}
    \FOR {$i \FROM n-1 \DOWNTO k$}
    \IF {$k=0$}
    \STATE $h[i]\gets i$ \ \COMMENT {initialise height function for column $0$}
    \ELSE
    \STATE $h[i]\gets h[i] - B[i,k-1]$
    \ \COMMENT {adapt height function to column $k$}
    \ENDIF
    \IF {$i<n-1$}
    \STATE $\clify{i}$
    \ENDIF \ENDFOR \ENDFOR
  \end{algorithmic}
\end{algorithm}

For this algorithm it is easy to see that in the inner loop for~$k$, the
condition of proposition~\ref{basic bijection} is satisfied, provided that the
paths $P_k,\ldots,P_{n-1}$ are disjoint at the start of the loop. Indeed the
condition when calling $\clify{i}$ precisely requires the disjointness of
$P_i$ and $P_{i+1}$, and although a preceding $\clify{\ip}$ may have changed
the entries $B_{i+1,j}$ that describe $P_{i+1}$, this can only have made them
smaller, moving $P_{i+1}$ away from $P_i$. In column~$k$ the vertical
steps introduced come \emph{before} the unchanging point where $P_{i+1}$
leaves that column, so this does not endanger disjointness with~$P_i$ either.
On the other hand it is not obvious that $P_{k+1},\ldots,P_{n-1}$ are again
disjoint at the end of the inner loop (and of course $P_k$ in general
\emph{will not} be disjoint from them). This brings us to the main technical
verification that needs to be done in order to conclude that we have described
well defined combing and uncombing bijections.

\begin{prop} \label{column bijection}
  Let $\Pathfam(n,k)$ denote the subset of $\Pathfam(n)$ of pairs $(B,D)$
  encoding $n$-families without any vertical steps in any non-final column
  before column~$k$ (so $D_{i,j}=0$ whenever $0\leq j<k$ and $j<i<n$) and for
  which the supports of the paths $P_k,\ldots,P_{n-1}$ are all disjoint. Then
  for each $k<n$, the inner loop at $k$ of procedure~\ref{combing} defines a
  bijection, and the one of procedure~\ref{uncombing} defines the inverse
  bijection, between $\Pathfam(n,\kp)$ and $\Pathfam(n,k)$
\end{prop}

\begin{proof}
  We have already seen that, when starting in the forward direction from an
  element of $\Pathfam(n,\kp)$, the calls $\disj{i}$ in the inner loop of
  procedure~\ref{combing} are invoked under the proper conditions: the number
  of units of~$D_{i,k}$ (vertical steps) that such a call transfers to
  $D_{i+1,k}$ does not exceed the value of $D_{i,k}$ at that point. Starting
  in the backward direction from an element of $\Pathfam(n,k)$, the inner loop
  of procedure~\ref{uncombing} will also invoke the calls of $\clify{i}$ under
  the proper conditions, and they will transfer all units from $D_{i+1,k}$ to
  $D_{i,k}$, so that in the end all units of column~$k$ of~$D$ have been
  combined into~$D_{k,k}$. The only point left to prove is the disjointness of
  the supports of the indicated set of paths at the completion of the inner
  loop, in both directions. This was assumed and remains unchanged beyond
  column~$k$, and for column~$k$ the verifications were done in
  proposition~\ref{basic bijection} (the disjointness in that column obtained
  by $\disj{i}$ is not endangered by a following $\disj{\ip}$). So only the
  parts of the paths in columns $j<k$ need to be considered.

  It is part of proposition~\ref{basic bijection} that after $\disj{i}$ the
  paths $P'_i$ and $P'_{i+1}$ have disjoint supports up to column~$k$, but (if
  $i\neq n-2$) the subsequent application of $\disj{\ip}$ may move $P'_{i+1}$
  in the direction of~$P'_i$ again, and we need to show that the resulting
  path $P''_{i+1}$ nevertheless stays disjoint from~$P'_i$. Let as before
  $h_0,h_1$ be the functions describing the initial paths $P_i$ and $P_{i+1}$,
  with $h'_0,h'_1$ the ones after modification by $\disj{i}$; let $h_2$
  similarly describe the initial path $P_{i+2}$, and call the functions
  obtained after $\disj{\ip}$ modifies $h'_1$ and~$h_2$ respectively $h''_1$
  and~$h'_2$. Just as $\disj{i}$ determines a sequence $(d_0,\ldots,d_k)$
  there is a sequence determined by $\disj{\ip}$ that we call
  $(e_0,\ldots,e_k)$; then one has equation (\ref{fprimedef}) and similarly
  $h''_1(j)=h'_1(j)-e_j$, and $h'_2(j)=h_2(j)+e_j$ for $0\leq j\leq k$. From
  (\ref{fprimeineq}) we have $h'_0(j)<h'_1(j)-d_j$ and we wish to show
  $h'_0(j)<h''_1(j)=h'_1(j)-e_j$. It will therefore suffice to show that
  $e_j\leq d_j$ for $0\leq j\leq k$. We shall do so by induction on~$j$; the
  starting case $e_0=0=d_0$ is trivial, so suppose $j>0$. Then the equivalent
  of (\ref{djdef}) for $e_j$ can be written
  $e_j=\max(e_{j-1},h'_1(j)+1-h_2(j))$. Now by induction $e_{j-1}\leq
  d_{j-1}\leq d_j$, while from the hypothesis $h_1(j)<h_2(j)$ that $P_{i+1}$
  and $P_{i+2}$ are initially disjoint we get
  $h'_1(j)+1-h_2(j)=d_j+h_1(j)+1-h_2(j)\leq d_j$ as well, so indeed $e_j\leq
  d_j$.

  Having shown that the inner loop at $k$ of procedure~\ref{combing} maps
  $\Pathfam(n,\kp)$ to $\Pathfam(n,k)$, we must also prove that conversely the
  inner loop at $k$ of procedure~\ref{uncombing} maps $\Pathfam(n,k)$ to
  $\Pathfam(n,\kp)$. The situation is a bit different, in that the
  disjointness of $P_{i+2}$ and $P_{i+1}$ that we need to show (for $k\leq
  i<n-2$) is first potentially destroyed by $\clify{\ip}$, and then must be
  restored by $\clify{i}$. We can use the same notation as above, but the
  hypotheses differ: we assume that $\clify{\ip}$ transforms $(h''_1,h'_2)$
  into $(h'_1,h_2)$ while producing (from right to left) a sequence
  $(e_0,\ldots,e_k)$, and then $\clify{i}$ transforms $(h'_0,h'_1)$ into
  $(h_0,h_1)$ while producing a sequence $(d_0,\ldots,d_k)$. Again the key
  point is establishing $d_j\geq{e_j}$ for $0\leq j\leq k$, since analogously
  to~(\ref{fprimeineq}) one has $h'_1(j)=h''_1(j)+e_j<h'_2(j)$, and so the
  condition $d_j\geq{e_j}$ will imply the desired inequality
  $h_1(j)=h'_1(j)-d_j<h'_2(j)-e_j=h_2(j)$. This time we use descending
  induction on $j$; the initial case $e_k\leq d_k$ is a consequence of the
  fact that $\clify\ip$ transfers all $e_k$ vertical steps of~$P'_{i+1}$ to
  $P'_i$, where they contribute to~$d_k$. In the induction step we use
  equation~(\ref{djdefx}) in the form $d_j=\min(d_{j+1},h'_1(j)-h'_0(j)-1)$,
  which allows us to prove $d_j\geq{e_j}$ in two parts, as before: by
  induction $d_{j+1}\geq e_{j+1}\geq e_j$, and since $h''_1(j)>h'_0(j)$ (the
  hypothesis that the original paths $P''_{i+1}$ and $P'_i$ are disjoint) one
  also has $h'_1(j)-h'_0(j)-1=h''_1(j)+e_j-h'_0(j)-1\geq e_j$. This completes
  the proof.
\end{proof}

We can now state our main result, a bijective version of theorem~\ref{path
  thm}.

\begin{thm}
  The algorithm of procedure~\ref{combing} defines a bijection, and the
  algorithm of procedure~\ref{uncombing} defines the inverse bijection,
  between on hand the set of cliff-shaped Schr\"oder $n$-families, encoded by
  the corresponding strictly lower triangular matrices~$B$ with entries
  in~$\{0,1\}$, and on the other hand the set of disjoint Schr\"oder
  $n$-families, encoded by the corresponding pairs~$(B,D)$.
\end{thm}

\begin{proof}
  After pairing each $B$ corresponding to a cliff-shaped $n$-family
  with the corresponding diagonal matrix $D$ with diagonal entries
  $D_{k,k}=k-\sum_{0\leq{j}<i}B_{i,k}$, the set of cliff-shaped Schr\"oder
  $n$-families corresponds to $\Pathfam(n,n)$ and the set of disjoint
  Schr\"oder $n$-families corresponds to $\Pathfam(n,0)$. Now
  procedure~\ref{combing} realises the composite map
  \begin{equation}
    \Pathfam(n,n)\to\Pathfam(n,n-1)\to\cdots\to\Pathfam(n,0)
  \end{equation}
  where the individual maps are the bijections of proposition~\ref{column
    bijection}, and procedure~\ref{uncombing} realises
  the reverse composition of the corresponding inverse bijections.
\end{proof}

It may be observed that the initial map $\Pathfam(n,n)\to\Pathfam(n,n-1)$ and
the final map $\Pathfam(n,1)\to\Pathfam(n,0)$ are in fact identity maps: the
sets of families involved are the same in both cases (with just slightly
different descriptions), namely that of the cliff-shaped $n$-families
respectively that of the disjoint $n$-families, and our procedures only perform
some administrative actions without any changes to the paths for $k=n-1$ and
for $k=0$.

\section{Some complements and discussion}

As we have mentioned in the introduction, and illustrated in
figure~\ref{path-tiling}, there is a bijection between disjoint $n$-families
and tilings of the Aztec diamond of order~$n-1$. It is not easy to attribute
the discovery of this bijection clearly: a bijection between families of paths
and domino tilings of the Aztec diamond is first mentioned in~\cite{EuFu}, in
the proof of their proposition~2.2; however it is strongly based on a
bijection involving single paths that occurs in a slightly different context,
and whose origin goes back to Sachs and Zernitz~\cite{SaZe}. That context is
originally that of counting dimer coverings (perfect matchings) in a graph
describing the adjacency of squares in the \emph{augmented} Aztec diamond,
obtained from an Aztec diamond of order~$n$ by replacing the $2\times 2n$
rectangle it contains by a $3\times 2n$ rectangle; each such covering
(equivalent to a domino tiling of the augmented Aztec diamond) turns out to be
determined (bijectively) by a path from source to sink in a particular
orientation of the graph that is illustrated in figure~\ref{aug-Aztec-graph}.
\begin{figure}
  \includegraphics[width=0.4\textwidth]{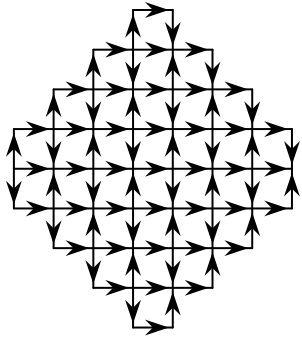}
  \caption{Directed graph for augmented Aztec diamond, order 4}
  \label{aug-Aztec-graph}
\end{figure}
The observation that those paths can be replaced by
paths with three types of steps, two of which are not parallel but at a
$45^\circ$-angle with the corresponding dominoes (our Schr\"oder-type paths),
is due to Dana Randall (unpublished), and is mentioned in \cite{Ciucu} and
\cite[p.~277 (6.49~a)]{EC2}.

The illustrations of this phenomenon, like our figure~\ref{path-tiling}, are
considered so convincing that one does not find in the references we cited
anything more precise than a rule how to associate a path family to a tiling,
with no attempt to formulate a proof of bijectivity of the correspondence.
Even though the proof is indeed straightforward, it is worth while to
formulate one, as this gives the occasion to see just how few assumptions
about the nature of the context are used, so that the argument can prove a
much more general statement. For this reason we give here such a statement and
its proof.

\begin{prop}
  Let $S$ be a finite subset of $\Z^2$, viewed as a set of squares in the
  plane, with $B=\setof{(i,j)\in S}{i\equiv j\pmod2}$ and $W=S\setminus B$ its
  subsets of black respectively white squares. Define sets $E,I,X$ of
  vertical edges with a white square~$w$ to their left and a black square~$b$
  to their right, where $E$ (the ``entries'') is the set of such edges with
  $w\notin W$ and $b\in B$, the set $X$ (the ``exits'') is that of such edges
  with $w\in W$ and $b\notin B$, and $I$ (the ``interior edges'') is the set
  of such edges with $w\in W$ and $b\in B$; formally (identifying an edge with
  the square to its right)
  \begin{align*}
    E&=\setof{b\in B}{b-(0,1)\notin W},\\
    I&=\setof{b\in B}{b-(0,1)\in W},\\
    X&=\setof{b\in \Z^2\setminus B}{b-(0,1)\in W}.
  \end{align*}
  There is a bijection between the set of domino tilings of $S$ and the set of
  families of paths, using steps chosen from
  $\{(1,1,),(0,2),(-1,1)\}$, such that each entry in $E$ is connected by some
  path to an exit in~$X$ and vice versa, with paths passing through elements
  of~$I$ only, and such that each element lies on at most one path.
\end{prop}

Finiteness is the only hypothesis made for the set of squares for which domino
tilings are considered (we leave it as an exercise to find where it is used
implicitly in the proof below). This means of course that very possibly no
domino tilings exist at all, and therefore no path families. The most obvious
obstruction against the existence of such tiling is a nonzero balance
$\#B-\#W$ between black and white squares; this balance is equal to the
balance $\#E-\#X$ between entry and exit points for the path which clearly
must be zero for path families to exist.

While the expressions for $E,I,X$ in the statement of the proposition identify
vertical edges with a black square to their right with that black square (as
an element of $\Z^2$), our proof be in a geometric language that distinguishes
them as different kinds of objects.

\begin{proof}
  Suppose first that a domino tiling of $S$ is given. We associate to each
  domino~$d$ of the tiling a pair $(e,e')\in(E\union I)\times(I\union X)$,
  which will serve as a step in one of the paths of the corresponding family
  whenever $e\neq e'$: we take $e$ to be the left edge of the black square
  of~$d$, and $e'$ is the right edge of the white square of~$d$. Every edge in
  $E\union I$ occurs as $e$ for a unique domino of the tiling, namely for the
  domino the contains the square $b\in B$ at the right of the edge, and every
  edge in $I\union X$ occurs as $e'$ for a unique domino of the tiling, namely
  for the domino the contains the square $w\in W$ at the left of the edge.
  According to the four possibilities for the orientation and colouring of a
  domino, each such pair $(e,e')$ either satisfies $e=e'$, or that $e'-e$ is
  in the set $\{(1,1,),(0,2),(-1,1)\}$ of allowed steps; therefore by
  collecting those pairs with $e\neq e'$ and chaining them together, we get a
  family of paths (cycles are of course impossible due to strict monotonicity
  of the second coordinate) that has all the stated properties.

  Conversely let a family of paths as described in the proposition be given.
  For any black square $b\in B$ of $S$, its left edge~$e$ belongs to $E\union
  I$; if $e$ is in $I$ but not on any path for the family, then the white
  square~$w$ to the left of~$e$ is in~$W$ and $(w,b)$ will form a domino of
  the tiling; otherwise $b$ will form a domino with the white square to the
  left of the edge~$e'$ reached from~$e$ by one forward step on the path
  passing through it. Similarly the right edge~$e'$ of any square $w\in W$
  belongs to $I\union X$, and $w$ is paired either with the black square the
  right of $e'$ if $e'\in I$ is not on any path for the family, or otherwise
  with the black square to the right of the edge~$e$ reached by going one step
  back along the path passing through~$e'$. Clearly this attribution of
  squares is reciprocal, so one obtains a partition of $S=B\union W$ into
  dominoes. The maps from domino tilings to path families and vice versa are
  inverses of  each other, by inspection of the definitions.
\end{proof}

We note that a similar result can be proved in the same way for lozenge
tilings of a subset of triangles in a triangular tiling of the plane, and
leads to a bijection with families of disjoint paths in which only two basic
steps are allowed.

To apply this proposition to obtain the correspondence between domino tilings
of the Aztec diamond of order~$n-1$ and disjoint Schr\"oder $n$-families, it
suffices to apply a linear transformation with matrix
$\frac12\smallmat1{-1}1{\phantom-1}$ to the paths, so as to map their basic
steps respectively to $(0,1)$, $(-1,1)$ and $(-1,1)$, and then shift them to
match the required starting and ending points. A small proviso must be made
for the path $P_0$ with $0$ steps (our proposition cannot produce such paths
due to $E\thru X=\emptyset$): we simply add this path in the proper place, on
the edge that sticks out beyond the two squares at a corner of the Aztec
diamond, which edge may be thought of as part of the configuration even though
neither of the squares it separates belong to the Aztec diamond.

The proposition allows us to understand the qualitative difference between the
problems of tiling the Aztec diamond and the augmented Aztec diamond: the
latter (if properly positioned) gives rise in the path setting to a situation
where there is just a single entry and a single exit, whereas for the Aztec
diamond of order~$n$ there are $n$ entries and $n$ exits.
\begin{figure}
  \includegraphics[width=\textwidth]{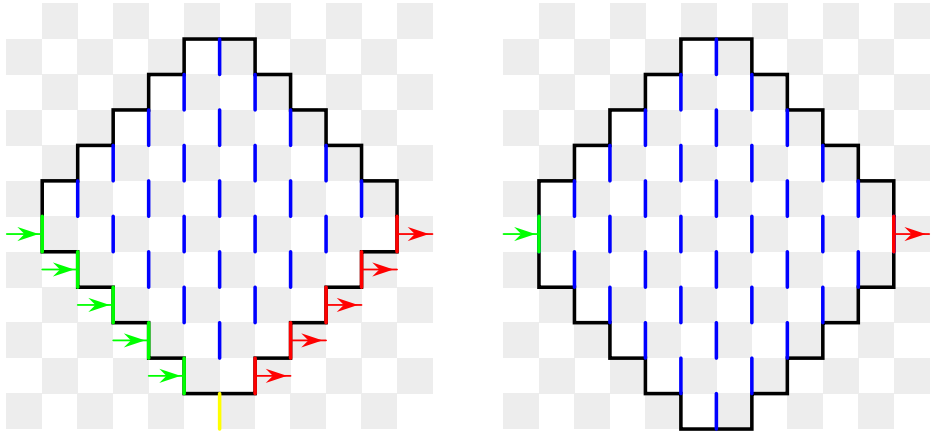}
  \caption{Aztec and augmented diamonds; entries and exits}
  \label{compare-Aztecs}
\end{figure}
This is illustrated in figure~\ref{compare-Aztecs}.

From the point of view of domino tilings, the choice to focus on vertical
edges between squares with a black square on their right is an arbitrary one
among four similar possibilities. This means that with one domino tiling one
can associate four different disjoint path families by making different
choices, adapting the direction of the basic steps in paths, as is illustrated
in figure~\ref{4fam}.
\begin{figure}
  \includegraphics[width=0.75\textwidth]{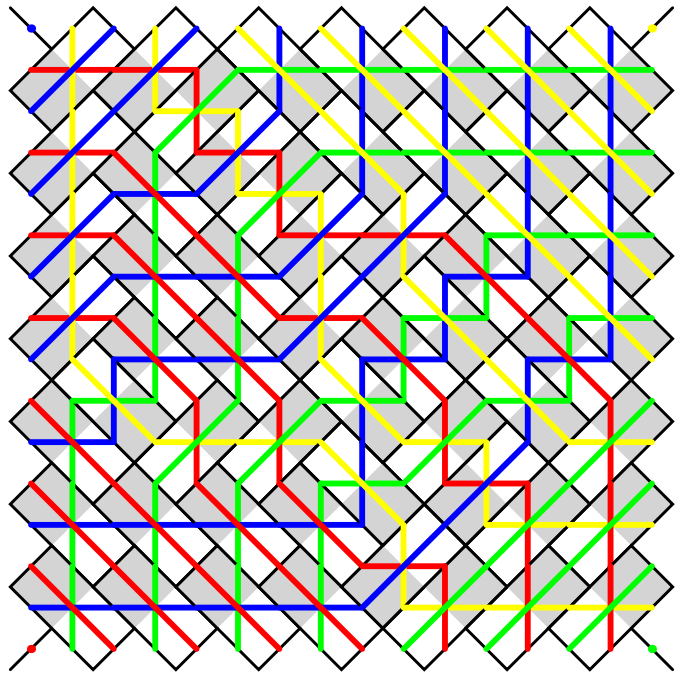}
  \caption{Four families of disjoint paths for a single tiling}
  \label{4fam}
\end{figure}
Note that the duality illustrated in figure~\ref{duality} just expresses the
relation between two of these disjoint path families associated to the same
domino tiling, those using the two possibilities with (after transformation)
SW--NE running edges.

In the introduction we mentioned ``domino shuffling'' as a previously known
method of constructing a domino tiling of the Aztec diamond of order~$n$ using
a sequence of $\frac{n(n+1)}2$ bits as input, in an invertible manner. In this
aspect our algorithm is similar to domino shuffling, but a closer comparison
show that the methods are nevertheless quite different.

In domino shuffling a tiling is obtained by constructing tilings Aztec
diamonds of increasing order until the desired order is attained. In passing
from one order to the next, a first step is to remove information from the
configuration (a number of ``bad blocks'' of two dominoes each are removed),
then the remaining dominoes are shifted by a fixed rule in the direction of
one of the corners, those corners themselves moving outwards so as to enlarge
the diamond, and finally the resulting open space is filled with a choice of
``good blocks'' of two dominoes each. The shifting rule simply looks at the
type of the domino as is apparent in figure~\ref{4fam}, and moves the domino
towards the corner which (in our figure) contains a domino of the same type;
good and bad blocks are pairs of dominoes that form a $2\times 2$ square with
a dark respectively light square in the leftmost corner (again in our figure).
Each such block has one of two possible tilings and therefore represents one
bit of information, so the net information that is added in the expansion from
order $i-1$ to order $i$ is the difference between the number of good blocks
added and the number of bad blocks removed; this amounts to~$i$ fresh bits,
independently of those individual numbers.

Although the information contained in the bad blocks can be recycled when
inserting good blocks, this repeated partial deconstruction/reconstruction
gives a certain irregularity of operation to domino shuffling that is an
essential aspect of it. The removal cannot be avoided, because the dominoes of
the bad block would get in the way of the others. The method is based on a
representation of tilings by a pair of alternating sign matrices which differ
by~$1$ in size, and each of which severely restricts the possibilities for
choosing the other; the removal of the bad blocks corresponds to forgetting
the smaller of these matrices while keeping the other, and the insertion of
good blocks to choosing a new alternating sign matrix one size larger than the
one that was kept, forming a new pair. Thus one works oneself up to ever
larger pairs, making sure to keep one matrix, and thereby the major part of
the accumulated information, intact at all times.

Our algorithm is completely different. For one thing, the possibly
intersecting families of lattice paths it operates upon do not correspond to
domino tilings at all. When part of the path family has been made disjoint,
this can be translated into an incomplete tiling with irregular border, but
while the integration of a new path into this part is done by a ``sweep''
iterating a procedure that is simple to describe in terms of paths, the
description of the corresponding ``ripple'' that modifies and extends the
incomplete tiling does not appear to be very easy. Finally, our procedure
treats paths by decreasing size, and so uses its bits grouped
$n,n-1,\ldots,1$, which is the opposite order as used in domino-shuffling.
This seems to be an essential aspect of our procedure; we cannot see how it
(or a variant) could be used to expand a disjoint $n$-family to a disjoint
$n+1$-family by integrating a new cliff-shaped path with $n$ bits of fresh
information.

We conclude by telling how our algorithm was found, which happened without
realising at first any connection with the Aztec diamond theorem. It started
with a question~\cite{Math.SE} posed on the online form Math.StackExchange. It
asked for an explanation of the nice evaluation of a the determinant of a
matrix with entries defined by a recurrence relation, a slight generalisation
of the matrix~$A_{[n]}$ of Delannoy numbers of section~2. One of the answers
given (by ``Grigory~M'') proposed a combinatorial explanation in terms of
counting families of non-intersecting lattice paths, but failed to complete
the argument showing that this enumeration was given by the proposed formula.
The second author, having came across this question and incomplete answer, was
also unable to find a combinatorial argument, but discussed the problem with
the first author. In this discussion various ideas were attempted, but no
bijective proof was found. The first author eventually came up with the
algorithm leading to bijection presented in this paper, but the details were
never discussed to the point of convincing the second author. several months
later that the second author learned, through a discussion with Christian
Krattenthaler, that this lattice path enumeration was known to be equivalent
to the well known Aztec diamond theorem, but that no quite satisfactory
bijective proof of it was known. This led to renewed interest and discussion
between the authors, in the course of which the details of the algorithm were
made sufficiently clear to be implemented in a computer program. This
dispelled any doubt about the validity of the method, and eventually led to
writing the current paper.

\end{document}